\documentclass[a4paper,12pt,reqno]{amsart}
\usepackage{mathrsfs}
\usepackage{amsfonts}
\usepackage{amssymb}
\usepackage{amsmath}
\usepackage{amsthm,color}
\usepackage{pifont}
\numberwithin{equation}{section}
\allowdisplaybreaks

\newtheorem{thm}{Theorem}[section]
\newtheorem{prop}[thm]{Proposition}
\newtheorem{lem}[thm]{Lemma}

\renewcommand{\theequation}
{\thesection.\arabic{equation}}

\renewcommand{\Re}{\operatorname{Re}}
\renewcommand{\Im}{\operatorname{Im}}

\def\mr{\mathbb{R}}

\def\Lp{L^{p}}
\def\L1{L^{1}}

\def\B0{B_{0}}

\begin{document}

\title[A note on generalized spherical maximal means]{A note on generalized spherical maximal means}
\author{Feng Zhang}

\address{Feng Zhang, School of Mathematical Sciences\\
	Xiamen University\\
	Xiamen 361005, People's Republic of China}

\email{fengzhang@stu.xmu.edu.cn}

\makeatletter
\@namedef{subjclassname@2020}{\textup{2020} Mathematics Subject Classification}
\makeatother
\subjclass[2020]{42B15, 42B20, 42B25}

\date{\today}

\keywords{generaliezed spherical maximal means, necessary condition}
\begin{abstract}
	The goal of this note is to provide an alternative proof of Theorem 1.1 (i) in \cite{Liu2023arXiv}, that is, if $n\geq 2$ and $M^{\alpha}$ is bounded on $L^{p}(\mathbb{R}^{n})$ for some $\alpha\in \mathbb{C}$ and $p\geq 2$, then we have 
	\begin{align*}
		\Re \alpha\geq \max\left\{\frac{1-n}{2}+\frac{1}{p},\frac{1-n}{p}\right\}.
	\end{align*}
   \end{abstract}

\maketitle

\section{Introduction}
In \cite{Stein1976Acad}, Stein introduced the generalized spherical maximal means
 $$M^{\alpha}f(x):=\sup_{t>0}|A_t^{\alpha} f(x)|,$$ 
where $A_t^{\alpha} f(x)$ is defined as
\begin{align*}
	A_t^{\alpha} f(x):=\frac{1}{\Gamma(\alpha)} \int_{|y| \leq 1}\left(1-|y|^2\right)^{\alpha-1} f(x-t y)dy.
\end{align*}
The generalized spherical means are defined a priori only for $\Re \alpha>0$. A direct calculation (see \cite[p.~171]{Stein1973book} and \cite[Appendix A]{miao2017Proc}) implies 
\begin{align}\label{eq-average-def}
	\widehat{A_{t}^{\alpha}f}(\xi)=\widehat{f}(\xi)\pi^{-\alpha+1}|t\xi|^{-n/2-\alpha+1}J_{n/2+\alpha-1}(2\pi|t\xi|)=:\widehat{f}(\xi)m^{\alpha}(t\xi),
\end{align}
where $J_{\beta}$ denotes the Bessel function of order $\beta$. Recall that for $\beta\in \mathbb{C}$ and $r>0$, the Bessel function $J_\beta(r)$ is given by
\begin{align*}
	J_\beta(r):=\sum_{j=0}^{\infty} \frac{(-1)^j}{j !} \frac{1}{\Gamma(j+\beta+1)}\left(\frac{r}{2}\right)^{2 j+\beta}.
\end{align*}
For more details we refer the readers to \cite[Chapter \uppercase\expandafter{\romannumeral2}]{bookWatson}.
Thus, the definition of $A_t^{\alpha} f$ can be extended to $\alpha\in \mathbb{C}$ via \eqref{eq-average-def}. In particular, one can recover the averages over Euclidean balls and the classical spherical means by taking $\alpha=1$ and $\alpha=0$, respectively. Stein \cite{Stein1976Acad} proved that $M^{\alpha}$ is bounded on $\Lp(\mathbb{R}^{n})$ if 
\begin{align}\label{eq-range-p1}
	1<p\leq 2 \text{ and } \Re \alpha>1-n+\frac{n}{p}
\end{align}
or
\begin{align}\label{eq-range-p2}
	2\leq p\leq \infty  \text{ and } \Re\alpha>\frac{2-n}{p}.
\end{align}
The conditions in \eqref{eq-range-p1} are optimal, see \cite[p.~519]{Stein93}. Applying this result, Stein showed that $M^{0}$ is bounded on $\Lp(\mathbb{R}^{n})$ whenever $p>n/(n-1)$ and $n\geq 3$. Later, Bourgain obtained the $\Lp(\mathbb{R}^{2})$-boundedness of $M^{0}$ for $p>2$ in \cite{Bou86JAM}. Based on the local smoothing estimate, Mockenhaupt, Seeger and Sogge \cite{Mockenhaupt1992Ann} provided an alternative proof of Bourgain's result and improved the range of $p$ in \eqref{eq-range-p2} when $n=2$. Miao, Yang and Zheng \cite{miao2017Proc} made a further improvement by using the Bourgain-Demeter $\ell^{2}$ decoupling theorem \cite{Bou15AnnMath}, that is, $M^{\alpha}$ is bounded on $L^{p}(\mathbb{R}^{n})$ if
\begin{align*}
	2\leq p\leq \frac{2(n+1)}{n-1} \text{ and } \Re \alpha>\frac{1-n}{4}+\frac{3-n}{2p}
\end{align*}
or
\begin{align*}
	\frac{2(n+1)}{n-1}\leq p \leq \infty \text{ and } \Re \alpha>\frac{1-n}{p}.
\end{align*}
In \cite{Nowak23CPAA}, Nowak, Roncal and Szarek obtained some optimal results for the generalized spherical maximal means on radial functions when $n\geq 2$ and $\alpha>(1-n)/2$. Recently, Liu, Shen, Song and Yan \cite{Liu2023arXiv} established the following necessary conditions for the $L^{p}(\mathbb{R}^{n})$-boundedness of $M^{\alpha}$ and showed these conditions are almost optimal when $n=2$.
\begin{thm}\label{thm-main}
	Let $n\geq 2$ and $p\geq 2$. If $M^{\alpha}$ is bounded on $L^{p}(\mathbb{R}^{n})$ for some $\alpha\in \mathbb{C}$, then one of the following conditions holds
	\begin{enumerate}
		\item $2\leq p\leq 2n/(n-1)$ and $\Re\alpha\geq (1-n)/2+1/p$; 
		\item $p\geq 2n/(n-1)$ and $\Re\alpha\geq (1-n)/p$.
	\end{enumerate}
\end{thm}

The aim of this note is to provide an alternative proof of Theorem \ref{thm-main}. To show Theorem \ref{thm-main}, Liu, Shen, Song and Yan \cite{Liu2023arXiv} tested $M^{\alpha}$ on some functions whose Fourier transform concentrate on the direction $\vec{e}_{1}:=(1,0,\cdots,0)$ to avoid the interference between $e^{it\sqrt{-\Delta}}f$ and $e^{-it\sqrt{-\Delta}}f$. In this note, we will test $M^{\alpha}$ on
\begin{align*}
	\widehat{f}_{\lambda}(\xi):=e^{-2\pi i|\xi|}\chi(\lambda^{-1}|\xi|)|\xi|^{i\Im \alpha}
\end{align*}
for $\alpha\in \mathbb{C}$ and large $\lambda$. The function $e^{-2\pi i|\xi|}$ allows us to obtain the main term of $A_{t}^{\alpha}f_{\lambda}^{\alpha}(x)$ for some $x$ and $t$. To show Theorem \ref{thm-main} (i), we employ the main idea in \cite{Jones08Trans}. Observe that
\begin{align}\label{eq-foureri-dsigma}
	\widehat{d \sigma}(\lambda x)=\int_{\mathbb{S}^{n-1}} e^{2 \pi i \lambda x \cdot \theta} d \sigma(\theta) 
\end{align}
is almost a constant when $\lambda|x|$ is small. Thus, we obtain the main term of $A_{1}^{\alpha}f_{\lambda}^{\alpha}(x)$ when $\lambda|x|$ is small, from which Theorem \ref{thm-main} (i) follows. For Theorem \ref{thm-main} (ii), we choose $2\leq |x| \leq 3$ and $t_{x}:=|x|+1$. By the asymptotic expansion of \eqref{eq-foureri-dsigma}, we find the main term of $A_{t_{x}}^{\alpha}f_{\lambda}^{\alpha}(x)$ and further deduce Theorem \ref{thm-main} (ii). 

Throughout this article, each different appearance of the letter $C$ may represent a different positive constant and is independent of the main parameters. We write $A\lesssim B$ if there is $C>0$ such that $A\leq CB$, and write $A\approx B$ when $A\lesssim B \lesssim A$. $\widehat{f} $ means the Fourier transform of $f$. We denote by $\chi_{E}$ the characteristic function of $E$ for any $E\subset \mathbb{R}^{n}$,.

\section{The proof of Theorem \ref{thm-main}}
In this section, we establish Propositions \ref{prop-1} and \ref{prop-2}, from which Theorem \ref{thm-main} follows immediately. We first recall the following asymptotic expansion for the Bessel function
\begin{align}\label{eq-Bessel}
	J_\beta(r)=r^{-1 / 2} e^{i r}\left[b_{0,\beta}+E_{1,\beta}(r)\right]+r^{-1 / 2} e^{-i r}\left[d_{0,\beta}+E_{2,\beta}(r)\right], \quad r \geq 1,
\end{align}
where $b_{0,\beta},d_{0,\beta}$ are suitable coefficients and $E_{1,\beta}(r),E_{2,\beta}(r)$ satisfy
\begin{align*}
	\left|\left(\frac{d}{dr}\right)^{N}E_{1,\beta}(r)\right|+\left|\left(\frac{d}{dr}\right)^{N}E_{2,\beta}(r)\right|\lesssim r^{-N-1}, r\geq 1
\end{align*}
for any $N\in \mathbb{N}$. The following fact (see \cite[p.~347]{Stein93}) will be useful in the proof,
\begin{align}\label{eq-dsigma}
	\widehat{d\sigma}(\xi)=2\pi |\xi|^{(2-n)/2}J_{(n-2)/2}(2\pi |\xi|).
\end{align}

We first show the following lemma.
\begin{lem}\label{lem-test-function}
	Suppose $\chi\in C_{c}^{\infty}(\mathbb{R})$, $\chi \equiv 1$ on $[3/4,5/4]$ and $\chi\geq 0$. For $\alpha\in \mathbb{C}$ and $\lambda>0$, define 
	\begin{align*}
		\widehat{f}_{\lambda}^{\alpha}(\xi):=e^{-2\pi i|\xi|}\chi(\lambda^{-1}|\xi|)|\xi|^{i\Im \alpha}.
	\end{align*}
Then for $p\geq 1$, 
\begin{align*}
	\|f_{\lambda}^{\alpha}\|_{\Lp (\mr^{n})}\lesssim \lambda^{(n+1)/2-1/p}.
\end{align*}
\end{lem}
\begin{proof}
	By a change of variable, we deduce
	\begin{align*}
		f_{\lambda}^{\alpha}(x)=\lambda^{n+i\Im \alpha}\int_{0}^{\infty}\int_{\mathbb{S}^{n-1}}e^{2\pi i r\lambda x\cdot\theta}d\sigma(\theta)e^{-2\pi ir\lambda}\chi(r)r^{n-1+i\Im \alpha}dr.
	\end{align*}
	Note that
	\begin{align}\label{eq-vartheta}
		\int_{\mathbb{S}^{n-1}}e^{2\pi ir\lambda x\cdot\theta}d\sigma(\theta)=\int_{\mathbb{S}^{n-1}}e^{2\pi ir\lambda |x|\vec{e}_{1}\cdot\theta}d\sigma(\theta):=\vartheta(\lambda|x|r)
	\end{align}
	is a smooth function. Thus, a simple integration-by-parts argument shows
	\begin{align}\label{eq-flambda-1}
		|f_{\lambda}^{\alpha}(x)|\leq C_{N}\lambda^{-N}
	\end{align}
when $|x|\leq \lambda^{-1}$.

    If $|x|\geq \lambda^{-1}$, by \eqref{eq-Bessel} and \eqref{eq-dsigma}, we have
    \begin{align}
    	|f_{\lambda}^{\alpha}(x)|
    	&\lesssim \lambda^{(1+n)/2}|x|^{(1-n)/2}\left(\left|\int_{0}^{\infty}e^{2\pi ir\lambda (|x|-1)}a^{+}(2\pi r\lambda |x|)\chi(r)r^{(n-1)/2+i\Im \alpha}dr\right|\right. \notag\\
    	&\quad \left.+\left|\int_{0}^{\infty}e^{-2\pi ir\lambda (|x|+1)}a^{-}(2\pi r\lambda |x|)\chi(r)r^{(n-1)/2+i\Im \alpha}dr\right|\right) \label{eq-flambda},
    \end{align}
 where $a^{\pm}$ are standard symbols of order $0$. There are now two subcases.

{\bf Case (i)} If $||x|-1|\geq \lambda^{-1}$, by \eqref{eq-flambda} and a simple integration-by-parts argument, we obtain
	\begin{align}\label{eq-flambda-2}
		|f_{\lambda}^{\alpha}(x)|\lesssim \lambda^{(n+1)/2}|x|^{(1-n)/2}(\lambda ||x|-1||)^{-N+(1-n)/2} \lesssim \lambda(\lambda ||x|-1||)^{-N}
	\end{align}
for any $N\in \mathbb{N}$, where in the second inequality we used 
\begin{align*}
	|x|||x|-1|\geq \frac14.
\end{align*}

{\bf Case (ii)}
	If $\left||x|-1\right|\leq \lambda^{-1}$, \eqref{eq-flambda} implies that
	\begin{align}\label{eq-flambda-3}
		|f_{\lambda}^{\alpha}(x)|\lesssim \lambda^{(n+1)/2}|x|^{(1-n)/2}.
	\end{align}

Thus, combining \eqref{eq-flambda-1}, \eqref{eq-flambda-2}, \eqref{eq-flambda-3} and the fact
\begin{align*}
	\frac{r}{\lambda(r-1)}\leq 1
\end{align*}
for $r\geq 0$ with $|r-1|\geq 1$, we have
	\begin{align*}
		\|f_{\lambda}^{\alpha}\|_{\Lp(\mr^{n})}\lesssim \lambda^{(1+n)/2-1/p}.
	\end{align*}
This finishes the proof.
\end{proof}

Now, we prove the following proposition.
\begin{prop}\label{prop-1}
	Let $\alpha\in \mathbb{C}$ and $p\geq 2$. If $M^{\alpha}$ is bounded on $L^{p}(\mathbb{R}^{n})$, then
    \begin{align*}
    	\Re \alpha\geq \frac{1-n}{p}.
    \end{align*}
\end{prop}
\begin{proof}
By \eqref{eq-Bessel} and \eqref{eq-vartheta}, we have
\begin{align*}
	A_{1}^{\alpha}f_{\lambda}(x)&=
	\pi^{-\alpha+1}\lambda^{n/2+1-\Re\alpha} \int_{0}^{\infty}\vartheta(\lambda r|x|)e^{-2\pi i\lambda r}\chi(r)J_{n/2+\alpha-1}(2\pi \lambda r)r^{n/2-\Re\alpha}dr \\
	&=\lambda^{(n+1)/2-\Re\alpha} \int_{0}^{\infty}\vartheta(\lambda r|x|)\chi(r)\left(b_{1,\alpha}+e^{-4\pi i\lambda r}d_{1,\alpha}+a_{1}(\lambda r)\right)r^{(n-1)/2-\Re\alpha}dr \\
	&=:\sum_{i=1}^{3}I_{i}(x,\lambda),
\end{align*}
where 
\begin{align*}
	\left|a_{1}(\lambda r)\right|\lesssim (\lambda r)^{-1}, \quad \lambda r\geq 1.
\end{align*}

Suppose $c_{0}$ is small enough and $|x|\leq c_{0}\lambda^{-1}$. Obviously, 
\begin{align}\label{eq-I3}
	\left|I_{3}(x,\lambda)\right|\lesssim \lambda^{(n-1)/2-\Re\alpha}.
\end{align}
It remains to estimate $I_{1}(x,\lambda)$ and $I_{2}(x,\lambda)$. Without loss of generality, we may assume $b_{1,\alpha}=d_{1,\alpha}=1$. Integrating by parts, we obtain
\begin{align}\label{eq-I2}
	\left|I_{2}(x,\lambda)\right|\lesssim \lambda^{-N}
\end{align}
for any $N\in \mathbb{N}$. 

For $I_{1}(x,\lambda)$, by the mean value theorem, we deduce
\begin{align*}
	\left|\int_{0}^{\infty}\left(\vartheta(\lambda r|x|)-1\right)\chi\left(r\right)r^{(n-1)/2-\Re \alpha}dr\right|\leq c_{0}C,
\end{align*}
from which it follows that
\begin{align}\label{eq-I1}
	|I_{1}(x,\lambda)|\geq C\lambda^{(n+1)/2-\Re \alpha}
\end{align}
By \eqref{eq-I3}-\eqref{eq-I1}, we conclude
\begin{align*}
	|A_{1}^{\alpha}f_{\lambda}(x)|\geq C\lambda^{(n+1)/2-\Re \alpha},
\end{align*}
which further implies
\begin{align*}
	\|A_{1}^{\alpha}f_{\lambda}\|_{\Lp(\mr^{n})}\geq C\lambda^{(n+1)/2-\Re \alpha-n/p}.
\end{align*}
Combining this fact with the assumption and Lemma \ref{lem-test-function}, we see
\begin{align*}
	\lambda^{(n+1)/2-\Re\alpha-n/p} \lesssim\lambda^{(n+1)/2-1/p}.
\end{align*}
Thus, the proof is complete.
\end{proof}

Finally, we consider $x\approx 2$ and obtain the following proposition.
\begin{prop}\label{prop-2}
	Let $\alpha\in \mathbb{C}$ and $p\geq 2$. If $M^{\alpha}$ is bounded on $L^{p}(\mathbb{R}^{n})$, then
	\begin{align*}
		\Re \alpha\geq \frac{1-n}{2}+\frac{1}{p}.
	\end{align*}
\end{prop}
\begin{proof}
	As in Proposition \ref{prop-1}, we have
	\begin{align*}
		A_{t}^{\alpha}f_{\lambda}^{\alpha}(x)=&\pi^{-\alpha+1}\lambda^{n/2+1-\Re \alpha}t^{-n/2-\alpha+1} \int_{0}^{\infty}\int_{\mathbb{S}^{n-1}}e^{2\pi i x\cdot \lambda r\theta}d\sigma(\theta)J_{n/2+\alpha-1}(2\pi t\lambda r) \\
		&e^{-2\pi i\lambda r}\chi\left(r\right)r^{d/2-\Re \alpha}dr. 
	\end{align*}
	It follows from \eqref{eq-Bessel} that
	\begin{align*}
		J_{(n-2)/2}(r)=r^{-1/2}e^{ir}c_{1}+r^{-1/2}e^{-ir}e_{1}+a_{2}(r),\quad r\geq 1,
	\end{align*}
	where
	\begin{align*}
		|a_{2}(r)|\lesssim r^{-3/2}, \quad r\geq 1.
	\end{align*}
	Based on this fact and \eqref{eq-dsigma}, for $|2\pi \lambda xr|\geq 1$, we deduce
	\begin{align*}
		\int_{\mathbb{S}^{n-1}}e^{2\pi i x\cdot \lambda r\theta}d\sigma(\theta)=(2\pi)^{1/2} |\lambda x r|^{(1-n)/2}\left(e^{i|2\pi\lambda x r|}c_{1}+e^{-i|2\pi\lambda x r|}e_{1}\right)+a_{3}(|2\pi\lambda xr|), 
	\end{align*}
	where
	\begin{align*}
		|a_{3}(|2\pi\lambda xr|)|\lesssim  |\lambda xr|^{-(n+1)/2}.
	\end{align*}
	Without loss of generality, we assume $c_{1}=e_{1}=1$.
	Similarly, 
	\begin{align*}
		J_{n/2+\alpha-1}(2\pi t\lambda r)=(2\pi t\lambda r)^{-1/2}\left(e^{i2\pi t\lambda r}c_{2,\alpha}+e^{-i2\pi t\lambda r}e_{2,\alpha}\right)+a_{4,\alpha}(2\pi t\lambda r),\quad 2\pi t\lambda r\geq 1,
	\end{align*}
	where
	\begin{align*}
		|a_{4,\alpha}(2\pi t\lambda r)|\lesssim \left(t\lambda r\right)^{-3/2}.
	\end{align*}
	For the sake of simplicity, we may assume $c_{2,\alpha}=e_{2,\alpha}=1$. Thus, we get
	\begin{align*}
		&\int_{\mathbb{S}^{n-1}}e^{2\pi i x\cdot \lambda r\theta}d\sigma(\theta)J_{n/2+\alpha-1}(2\pi t\lambda r)\\
		&\quad =(t\lambda r)^{-1/2}|\lambda xr|^{(1-n)/2}\left(e^{i2\pi \lambda r(|x|+t)}+e^{i2\pi \lambda r(|x|-t)}+e^{i2\pi \lambda r(-|x|+t)}+e^{i2\pi \lambda r(-|x|-t)}\right)\\
		&\quad \quad +a_{5,\alpha}(t,\lambda,|x|,r),
	\end{align*}
	where
	\begin{align*}
		|a_{5,\alpha}(t,\lambda,|x|,r)|\lesssim \lambda^{-(n+2)/2}.
	\end{align*}
	Hence,
	\begin{align*}
		A_{t}^{\alpha}f_{\lambda}^{\alpha}(x)&=\pi^{-\alpha+1}\lambda^{n/2+1-\Re \alpha}t^{-n/2-\alpha+1}\int_{0}^{\infty}[(t\lambda r)^{-1/2}|\lambda xr|^{(1-n)/2}\big(e^{i2\pi \lambda r(|x|+t)}+e^{i2\pi \lambda r(|x|-t)}\\
		&\quad  +e^{i2\pi \lambda r(-|x|+t)}+e^{i2\pi \lambda r(-|x|-t)}\big)+a_{5,\alpha}(t,\lambda,|x|,r)]e^{-i2\pi\lambda r}\chi\left(r\right)r^{n/2-\Re \alpha}dr \\ 
		&=:\sum_{i=1}^{5}I_{i}(x,t,\lambda).
	\end{align*}
	
	For $2\leq |x|\leq 3$, we choose $t_{x}:=|x|+1$. For $I_{5}(x,t_{x},\lambda)$, we have
	\begin{align*}
		|I_{5}(x,t_{x},\lambda)|&\lesssim \lambda^{n/2+1-\Re \alpha}\int_{0}^{\infty}\lambda^{-(n+2)/2}\chi(r)r^{n/2-\Re \alpha}dr \lesssim \lambda^{-\Re \alpha}.
	\end{align*}
	Note that the phase functions of $I_{1}(x,t_{x},\lambda), I_{2}(x,t_{x},\lambda)$ and $I_{4}(x,t_{x},\lambda)$ do not have critical points, which implies
	\begin{align*}
		|I_{i}(x,t_{x},\lambda)|\lesssim \lambda^{-N}, \quad i=1,2,4.
	\end{align*}

For the main term $I_{3}(x,t_{x},\lambda)$, we obtain
	\begin{align*}
		\left|I_{3}(x,t_{x},\lambda)\right|=\pi^{-\Re \alpha+1}\lambda^{1-\Re \alpha}t_{x}^{(1-n)/2-\Re  \alpha}|x|^{(1-n)/2}\int_{0}^{\infty}r^{-\Re \alpha}\chi(r)dr\geq C_{1}\lambda^{1-\Re \alpha}
	\end{align*}
for some $C_{1}>0$. Thus, we deduce
\begin{align*}
	|A_{t_{x}}^{\alpha}f_{\lambda}^{\alpha}(x)|\geq C_{2}\lambda^{1-\Re\alpha}
\end{align*}
when $\lambda$ is large enough. By the assumption and Lemma \ref{lem-test-function}, we conclude
	\begin{align*}
		\frac{1}{p}\leq \frac{n-1}{2}+\Re \alpha.
	\end{align*}
	This completes the proof. 
\end{proof}

\end{document}